\documentclass[12pt]{amsart}

\usepackage{amssymb}

\makeatletter
\@addtoreset{equation}{section}
\makeatother

\theoremstyle{plain}
\newtheorem{theorem}{Theorem}[section]
\newtheorem{lemma}[theorem]{Lemma}

\newtheorem{corollary}[theorem]{Corollary}

\theoremstyle{definition}

\newtheorem{remark}[theorem]{Remark}
\newtheorem{example}[theorem]{Example}

\newcommand{\Zee}{\mathbb{Z}}

\newcommand{\Que}{{\mathbb Q}}
\newcommand{\Oh}{{\mathbb O}}

\newcommand{\fB}{\mathcal{B}}
\newcommand{\fE}{\mathcal{E}}
\newcommand{\fH}{\mathcal{H}}
\newcommand{\fN}{\mathcal{N}}
\newcommand{\fU}{\mathcal{U}}

\newcommand{\ffn}{\mathfrak{n}}
\newcommand{\ffs}{\mathfrak{s}}

\newcommand{\alp}{\alpha}
\newcommand{\sig}{\sigma}
\newcommand{\vpi}{\varpi}

\newcommand{\HS}{\mathcal{HS}}
\newcommand{\topu}{\mathcal{T}_u}
\newcommand{\ctopu}{\overline{\mathcal{T}}_u}
\newcommand{\ze}{\mathrm{ZE}}
\newcommand{\ee}{\mathrm{E}}

\newcommand{\wbar}[1]{\overline{#1}}
\newcommand{\what}[1]{\widehat{#1}}

\newcommand{\wtil}[1]{\widetilde{#1}}

\newcommand{\am}{\mathrm{AM}}
\newcommand{\fal}{\mathrm{A}}
\newcommand{\fsal}{\mathrm{B}}
\newcommand{\bl}{\mathrm{L}}
\newcommand{\bm}{\mathrm{M}}
\newcommand{\ideal}{\mathrm{I}}

\newcommand{\pos}{\mathrm{P}}
\newcommand{\wstar}{\mathrm{W}^*}

\begin{document}

\author{Nico Spronk}
\address{Dept.\ of Pure Mathematics, University of Waterloo, Waterloo, Ontario, N2L 3G1, Canada}
\email{nspronk@uwaterloo.ca}

\title[Operator amenability of $\fsal(G)$]{On operator amenability of Fourier-Stieltjes algebras}

\begin{abstract}
We consider the Fourier-Stietljes algebra $\fsal(G)$ of a locally compact group $G$.
We show that operator amenablility of $\fsal(G)$ implies that a certain 
semitolpological compactification of $G$
admits only finitely many idempotents.  In the case that $G$ is connected, we show that
operator amenability of $\fsal(G)$ entails that $G$ is compact.
\end{abstract}

\subjclass{Primary 43A30; Secondary 46J40, 46L07, 43A07}

\thanks{The author was partially supported by an NSERC Discovery Grant.}


 \date{\today}
 
 \maketitle
 

\subsection{History and context}
Let $G$ be a locally compact group.  The Fourier-Steiltjes and Fourier algebras, $\fsal(G)$ and
$\fal(G)$ are defined by Eymard~\cite{eymard}, to act as dual objects of the measure 
and group algebras, $\bm(G)$ and $\bl^1(G)$, in a sense generalizing Pontryagin duality
from the theory of abelian locally compact groups.  Hence there is a natural expectation
that properties of the latter two algebras ought to be reflected in the former.

For example, consider the following comparison of certain amenability properties.

\begin{theorem}\label{theo:johruan}
The following are equivalent:
\begin{itemize} 
\item[(i)] the locally compact group $G$ is amenable;
\item[(ii)] {\rm (Johnson~\cite{johnson})} the (completely contractive) 
Banach algebra $\bl^1(G)$ is (operator) amenable;
\item[(iii)] {\rm (Ruan~\cite{ruan})} the completley contractive Banach algebra
$\fal(G)$ is operator amenable. 
\end{itemize}
\end{theorem}

We shall only note in passing that since $\bl^1(G)$ is naturally a maximal operator space, the relevant
properties discussed above are naturally ones of operator spaces.  For $\fal(G)$, this is not the case.
Forrest and Runde~\cite{forrestr} established that $\fal(G)$ is amenable only if $G$ is virtually abelian;
while the easier converse implication was noted by Lau, Loy and Willis~\cite{laulw}.  
Weak amenability of $\fal(G)$ is almost, but not completly, understood: see \cite{forrestr} and \cite{leelss}.
Operator weak amenability is completely understood:  thanks to Johnson~\cite{johnson1}, respectively
the second named author~\cite{spronk1}, and independently Samei \cite{samei}, we have
that both $\bl^1(G)$ and $\fal(G)$ are always operator weakly amenable.

Turning to the more formidable
measure algebras, we have the following.

\begin{theorem}\label{theo:dalsgh}
{\rm (Dales, Ghahramani and Helemski\u{\i}~\cite{dalesgh})}
For a locally compact group $G$ and the (completley contractive) Banach algebra $\bm(G)$ we have that:
\begin{itemize} 
\item[(i)] $\bm(G)$  is (operator) weakly amenable if and only if $G$ is discrete,
and even admits a point derivation when $G$ is non-discrete; and
\item[(ii)] $\bm(G)$ is (operator) amenable if and only if $G$ is discrete and amenable.
\end{itemize}
\end{theorem}

Like with $\bl^1(G)$, $\bm(G)$ is naturally a maximal operator space.  
Ghandehari~\cite[Theo.\ 5.2]{ghandehari}
noted that $\fsal(G)$ is amenable if and only if $G$ is compact and virtually abelian.
The considerations of Pontryagin duality, and Theorems \ref{theo:johruan} and \ref{theo:dalsgh}, at first blush,
suggest that 
\[
\mbox{{\it $\fsal(G)$ is operator amenable if and only if $G$ is compact.}}\tag{$\dagger$}
\]
This suggestion is false, in general.

\begin{theorem}\label{theo:rundes}
{\rm (Runde and S. \cite{rundes1})}
Each of the groups $G_p=\Que_p\rtimes\Oh_p^\times$ ($p$-adic numbers being acted upon by 
the multiplicative subgroup of $p$-adic integers) satisfy that
\begin{itemize} 
\item[(i)] $\fsal(G_p)$ is operator amenable; and
\item[(ii)] $\fsal(G_p)$ is weakly amenable.
\end{itemize}
\end{theorem}

Since $G_p$ is non-compact, this removes hope for the maintaining  naive conjectures about
operator amenability, operator weak amenability, or of weak amenability, for the algebra $\fsal(G)$.  
The major result of this note is that ($\dagger$) holds for connected groups.

\subsection{Structure of this paper}
The main aim of the present note
is the shed light on operator amenability of $\fsal(G)$.  
The two main results are Theorems \ref{theo:zeEcfin} and \ref{theo:BGoaGconn}.
The main two sections are essentaiilly development for these theorems.

\subsection{Background and notation}\label{ssec:bandn}
As mentioned before, the {\it Fourier-Stieltjes  algebra} $\fsal(G)$ is defined
by Eymard~\cite{eymard}.  This is the space of matrix coefficients of continuous unitary representations.  
This is the predual  of the universal von Neumann algebra $\wstar(G)$
and, as such, gains a natural operator space structure, as first noted by Blecher~\cite{blecher}.
See, also, the book of Effros and Ruan~\cite{effrosrB}.  Moreover, this is a {\it completely contractive Banach
algebra} in the sense that multiplication is a completely contractive bilinear map, or, equivalently,
the multiplication factors through the operator projective tensor products, $\fsal(G)\what{\otimes}\fsal(G)$.
We shall need few of details of this structure, but will be very reliant on some of its functorial properties, which
we discuss, below.

Following Johnson~\cite{johnson0,johnson}, Ruan~\cite{ruan} defined a completely contractive Banach algebra $\fB$
to be {\it operator amenable} if it operator projective tensor product $\fB\what{\otimes}\fB$ admits a 
{\it completely bounded approximate diagonal}:  a bounded net $(d_\alp)$ which satisfies
\[
(a\otimes 1)d_\alp-d_\alp(1\otimes a)\overset{\alp}{\longrightarrow}0\text{ and }
m(d_\alp)a,am(d_\alp)\overset{\alp}{\longrightarrow}a
\]
for $a$ in $\fB$, where $m$ is the multliplication map, and $a\otimes 1$ and $1\otimes a$ are understood to live in the
unitization of $\fB$ operator tensor itself, if $\fB$ is non-unital.  This gives a certain {\it operator amenability
constant}:
\[
\am_{op}(\fB)=\inf\left\{\sup_\alp\|d_\alp\|_{\fB\what{\otimes}\fB}:
\begin{matrix} (d_\alp)\subset\fB\what{\otimes}\fB
\text{ is a completley } \\ \text{bounded approximate diagonal}\end{matrix}\right\}.
\]
The following property is easy.
\begin{itemize}
\item[(a)] {\it If $\fB$ is operator amenable and $I$ is a closed ideal of $\fB$, then $\fB/I$ is operator amenable,
with $\am_{op}(\fB/I)\leq \am_{op}(\fB)$.}
\end{itemize}
Indeed, it follows from functorial properties of the operator projective tensor product.
The next uses property ideas which go back to Helemski\u{\i}~\cite{helemskii}, and are done more succinctly
by Curtis and Loy~\cite{curtisl}.  See also the text of Runde~\cite{rundeB}.
These have been adapted to the operator space setting by various
authors; see, for example, \cite{forrestw,forrestkls}.
\begin{itemize}
\item[(b)] {\it If $\fB$ is operator amenable and $I$ is a closed ideal of $\fB$ which is spatially completely complemented,
i.e.\ there is a completely bounded projection $P:\fB\to I$, then $I$ admits a bounded approximate identity and
is operator amenable.}
\end{itemize}
Amenability is similar to operator amenability, except that we consider only bounded maps
and the projective tensor product.  As noted above, $\fsal(G)$ is amenable only for compact
virtually abelian groups.

(Operator) amenability implies a certain formally and actually weaker property of {\it (operator) weak 
amenability}.  Furthermore, weak amenability implies operator weak amenability.
For commutative algebras, operator weak amenability enjoys our property (a),
but, at least to the authors' knowledge, it does not pass to complemented ideals.  
It is for this reason we shall have little to say about this property, and thus shall
not go into details of its definition.  We note only that (operator) weak amenability precludes
the existence of point derivations.

Let us make note of two spatially completely complemented ideals of $\fsal(G)$.
The {\it Fourier algebra}, $\fal(G)$, defined in \cite{eymard}, is the space of matrix coefficients
of the left regular representation.  The {\it Rajchman algebra} $\fsal_0(G)$ is the space
of elements of $\fal(G)$ vanishing at infinity.  This has been extensively studied by the first named 
author~\cite{ghandehari}.  Each is a closed, translation invariant subspace, hence admitting a central
projection in $\wstar(G)$ which, by the predual action, 
multiplies $\fsal(G)$ into this algebra, and thus is completely spatially complemented.  
Furthermore, each is an ideal.  We have that $\fal(G)\subseteq\fsal_0(G)$.

In the case that $G$ is abelian, we have that $\fsal(G)=\bm(\what{G})$ and
$\fal(G)=\bl^1(\what{G})$, algebraically isometrically (completley isometrically when
maximal operator space structure is applied to the latter two algebras) by way of the Fourier-Stieltjes
and Fourier transforms.  For pairs of general $\bl^1$-spaces, operator projective
tensor product is the usual projective tensor product.  We can drop the adjectives
`completely' and `operator' from our terminology, in this case.

\section{Topologies and idempotents}

We shall make essential use of
some results from recent work of the second-named author~\cite{spronk}.
Let $G$ be a locally compact group and denote by $\tau_G$ the topology on $G$.
Since we will consider other topologies, we shall think of $\tau_G$ as the ambient topology.

Let $\tau\subseteq\tau_G$ be a (not necessarily Hausdorff) topology for which
$(G,\tau)$ is a topological group.  We let $\pos^\tau$ denote the set of
$\tau$-continuous positive definite functions on $G$, and $\sig(G,\pos^\tau)$
the coarsest topology on $G$ making each member of $\pos^\tau(G)$ continuous.
We then define the {\it unitarizable topologies} on $G$ by
\[
\topu(G)=\{\tau\subseteq\tau_G:(G,\tau)\text{ is a topological group, and }\tau=\sig(G,\pos^\tau)\}.
\]
A variant of the GNS construction (see, for example, \cite{spronk})
shows that for each $\tau$ in $\topu(G)$, there
is a unitary representation $\vpi_\tau:G\to\fU(\fH_\tau)$ for which
induces a homeomorphism from $G/\wbar{\{e\}}^\tau$, with topology induced by $\tau$, onto $\vpi^\tau(G)$.

In particular we let $\vpi:G\to\fU(\fH_G)$ denote the universal representation of $G$, which is injective
and allows that $G$ is homeomorphic to $(\vpi(G),wo)$, where $wo$ denotes the weak operator topology.
This shows that $\tau_G\in\topu(G)$.
We let 
\[
G^\varpi=\wbar{\vpi(G)}^{wo}
\] 
in the closed ball of $\fB(\fH_G)$.  Then $G^\vpi$ is easily
seen to be a compact semi-toplogical semigroup in the weak operator topology.
We call this the {\it Eberlein compactification}.
We note the easy fact, shown, for example, in \cite{spronks}, but also noted elsewhere, 
that $G^\vpi$ is the Gelfand spectrum of the {\it Eberlein algebra}, $\fE(G)=\wbar{\fsal(G)}^{\|\cdot\|_\infty}$.

We consider the central idempotents
\[
\ze(G^\vpi)=\{p\in G^\vpi:p^2=p\text{ and }\vpi(s)p=p\vpi(s)\text{ for all }s\text{ in }G\}.
\]
Since all elements of $G^\pi$ are contractions, these idempotents are, in fact, projections.
The set $\ze(G^\vpi)$ has the usual partial ordering, $p\leq p'$ if and only if $pp'=p$.
The set $\topu(G)$ is partially ordered by $\subseteq$.

We summarize the role of these elements.

\begin{theorem}
{\rm (S. \cite{spronk})}
\begin{itemize}
\item[(i)] There are two partial order preserving maps 
\[
P:\topu(G)\to\ze(G^\vpi)\text{ and }T:\ze(G^\vpi)\to\topu(G) 
\]
which satisfy that $P\circ T=\mathrm{id}_{\ze(G^\vpi)}$
and $\tau\subseteq T\circ P(\tau)$ for each $\tau$ in $\topu(G)$.  Hence
$T\circ P$ is a closure operation on $\topu(G)$ and we let 
\[
\ctopu(G)=T\circ P(\topu(G)).
\]
Furthermore, if $\tau\subsetneq\tau'$ in $\ctopu(G)$, then every $\tau'$-compact set has
empty $\tau$-interior.
\item[(ii)] Using the preadjoint action of $\wstar(G)$ on $\fsal(G)$ we have for
$\tau\in\ctopu(G)$ that
\begin{align*}
&P(\tau)\cdot\fsal(G)=\{u\in\fsal:u\text{ is }\tau\text{-continuous}\}\text{, and} \\
&(I-P(\tau))\cdot\fsal(G)\text{ is an ideal of }\fsal(G).
\end{align*}
Hence letting $\fsal^\tau=P(\tau)\cdot\fsal(G)$ and $\ideal^\tau=(I-P(\tau))\cdot\fsal(G)$
we have a decomposition
\[
\fsal(G)=\fsal^\tau\oplus_{\ell^1}\ideal^\tau
\]
into a closed translation-invariant subalgebra and a closed translation-invariant ideal.
\end{itemize}
\end{theorem}

\begin{example}(i)  We have that $P(\tau_G)=I=\vpi(e)$.  Hence
$\fsal^{\tau_G}=\fsal(G)$ and $\ideal^{\tau_G}=\{0\}$.

(ii) If $\pos_{fin}(G)$ is the closed cone in the continuous positive definite functions generated by 
continuous finite dimensional unitary representations, then 
the {\it alomost periodic} or {\it Bohr} topology $\tau_{ap}=\sig(G,\pos_{fin}(G))$
is the smallest element of $\ctopu(G)$.  In fact, if $\tau$ in $\topu(G)$ is precompact, then
$T\circ P(\tau)=\tau_{ap}$.  We have that $\fsal^{\tau_{ap}}$ is the almost periodic
part of $\fsal(G)$, whereas $\ideal^{\tau_{ap}}$ is the ideal of matrix coefficients
of ``purely infinite" representations, i.e.\ those admitting no finite-dimensional subrepresentations.
The decomposition $\fsal(G)=\fsal^{\tau_{ap}}\oplus_{\ell^1}\ideal^{\tau_{ap}}$ was used critically
in~\cite{rundes,rundes1}.
\end{example}

\begin{remark}\label{rem:rajchmann}
The last comment in (i) of the theorem tells us that if $\tau\in\ctopu(G)\setminus\{\tau_G\}$, then
$\fsal_0(G)\subseteq \ideal^\tau$.  Hence if we define
\[
\ideal(G)=\bigcap_{\tau\in\ctopu(G)\setminus\{ \tau_G\}}\ideal^\tau
\]
then $\fsal_0(G)\subseteq \ideal(G)$.
\end{remark}

We consider a semilattice operation on $\topu(G)$.  Given $\tau,\tau'$ in $\topu(G)$, we see that
$\tau\cap\tau'$ is a group topology (consider inverse images of open sets under the product map
and the fact that inversion is a homeomorphism in each topology).  Furthermore
\[
\tau\cap\tau'=\sig(G,\pos^\tau)\cap\sig(G,\pos^{\tau'})=\sig(G,\pos^\tau\cap\pos^{\tau'})
=\sig(G,\pos^{\tau\cap\tau'}).
\]
Tarski's fixed point theorem (see Gierz et al~\cite[O-2.3]{gierzhklms}) shows that 
$\ctopu(G)$, being the set of fixed points of a monotone closure operator
$T\circ P$, is closed under the binary operation $\cap$.  

Let $L$ be a finite $\cap$-subsemilattice of $\ctopu(G)$.  Given $\tau_0\in L$ let
$L_{\tau_0}=\{\tau\in L:\tau\supseteq\tau_0\}$, which is a $\cap$-subsemilattice of $L$
which is {\it hereditary} in $L$ in the sense that for $\tau$ in $L_{\tau_0}$ and $\tau'$ in $L$
the relation $\tau\subseteq \tau'$ implies that $\tau'\in L_{\tau_0}$.  We notice that
for $\tau_0,\tau_1$ in $L$ then $L_{\tau_0}\cap L_{\tau_1}$,  if it is 
non-empty, admits a minimal element, hence
\[
L_{\tau_0}\cap L_{\tau_1}
=\begin{cases}L_{\tau_0\cap\tau_1} &\text{if one of }\tau_0,\tau_1\text{ is contained in the other} \\
\varnothing &\text{otherwise.}\end{cases}
\]
We then let
\[
\HS(L)=\{L_{\tau_0}:\tau_0\in L\}\cup\{\varnothing\}
\]
which is easily seen to contain all of the hereditary $\cap$-subsemilattices of $L$, and
itself is a $\cap$-semilattice of subsets of $L$.  It is easy to see that this is the dual to $(L,\cap)$ is the sense
that the set of indicator functions on $L$, $\{1_H:H\in\HS(L)\}$, is the set of semicharacters on $(L,\cap)$.

We now show that any finite  $\cap$-subsemilattices of $\ctopu(G)$ gives
a decomposition of $\fsal(G)$ as a direct sum of subalgebras ``graded" over
that subsemilattice's dual semilattice.  Notice that a consequence of (ii), below, is
that each constituent of the direct sum in (i) is a subalgebra.

\begin{lemma}\label{lem:Ldecomp}
Given a finite $\cap$-subsemilattice $L$ of $\ctopu(G)\setminus\{\tau_G\}$, and $H$ in $\HS(L)$, let
\[
P_L(H)=\prod_{\tau'\in H}P(\tau')\prod_{\tau\in L\setminus H}(I-P(\tau))\in\wstar(G)
\]
(where we take an empty product to be $I$) and $\fal_L(H)=P_L(H)\cdot \fsal(G)$. Then we have that
\begin{itemize}
\item[(i)] $\displaystyle \fsal(G)=\ell^1\text{-}\bigoplus_{H\in\HS(L)}\fal_L(H)$;
\item[(ii)] for $H,H'$ in $\HS(L)$, $\fal_L(H)\fal_L(H')\subseteq \fal_L(H\cap H')$; and
\item[(iii)] the map $Q_L:\fsal(G)\to\ell^1(\HS(L))$ (semilattice algebra) given by
\[
Q_L(u)=\sum_{H\in\HS(L)}[P_L(H)\cdot u(e)]\delta_H
\]
is a complete quotient homomorphism.
\end{itemize}
\end{lemma}

\begin{proof}
 It is evident that $P_L(\varnothing)=0$ if
$\tau_G\in L$, which is why we exclude that case.
Notice, too, that $P_L(L)=P(\bigcap L)$, where $\bigcap L$ is the smallest element of $L$,
since $P:\ctopu(G)\to\ze(G^\vpi)$ is increasing.  Furthermore, we see that $P(\tau')(I-P(\tau))=0$,
if $\tau'\subseteq\tau$ in $L$.  Hence we have that
\begin{align*}
I&=\prod_{\tau\in L}[P(\tau)+(I-P(\tau))] \\
&=\sum_{H\in \HS(L)}\prod_{\tau'\in H}P(\tau')\prod_{\tau\in L\setminus H}(I-P(\tau)) 
=\sum_{H\in \HS(L)}P_L(H)
\end{align*}
which gives (i)

We now note for $H$ in $\HS(L)$ that
\[
\fal_H(L)=\bigcap_{\tau'\in H}\fsal^{\tau'}\cap\bigcap_{\tau\in L\setminus H}\ideal^\tau
\]
where we treat the intersection over an empty index set as all of $\fsal(G)$.
Indeed, each space in the equation above is exactly
\[
\{u\in\fsal(G):P(\tau')\cdot u=u\text{ for }\tau'\text{ in }H\text{ and }
P(\tau)\cdot u=0\text{ for }\tau\text{ in }L\setminus H\}.
\]
Hence if $u\in \fal_L(H)$ and $u'\in\fal_L(H')$, then $uu'$ is $\tau'$-continuous for each
$\tau'$ in $H\cap H'$, and in the ideal $\ideal^\tau$ for each $\tau\in (L\setminus H)\cup
(L\setminus H')=L\setminus(H\cap H')$; i.e.\ $uu'\in \fal_L(H\cap H')$, which is (ii).

The multiplication relation (ii), tells us that $(P_L(H)\cdot u)(P_L(H')\cdot u')\in P_L(H\cap H')\cdot \fsal(G)$,
which implies that $Q_L$ is a homomorphism.  The adjoint map $Q_L^*:\ell^\infty(\HS(L))\to\wstar(G)$
is given by $Q_L^*(a_H)_{H\in\HS(L)}=\sum_{H\in \HS(L)}a_HP_L(H)$ is an isomorphism
of (commutative) C*-algebras, hence a complete isometry, which shows that $Q_L$ is a complete
quotient map.  \end{proof}

Now for our first main result.

\begin{theorem}\label{theo:zeEcfin}
If $\fsal(G)$ is operator amenable, then $\ze(G^\varpi)$ must be finite.
\end{theorem}

\begin{proof} 
If $\ze(G^\varpi)$ is infinite, then it admits arbitrarily large finite $\cap$-subsemilattices.
Given such a subsemilattice $L\subseteq\ctopu(G)\setminus\{\tau_G\}$, the previous lemma provides
the following estimates of amenability constants:
\begin{align*}
\mathrm{AM}_{op}(\fsal(G))&\geq \mathrm{AM}_{op}(\ell^1(\HS(L)))=\mathrm{AM}(\ell^1(\HS(L))) \\
&\geq 2|\HS(L)|-1\geq 2|L|+1.
\end{align*}
Indeed, the equality holds as $\ell^1$ is a maximal operator space, and the
middle inequality is shown in \cite[Cor.\ 1.8]{ghandeharihs} (and implicitly in \cite{duncann}).
\end{proof}

\begin{remark}\label{rem:FSfinidem}
Let us briefly examine the nature of $\fsal(G)$ when $\ctopu(G)$ is finite.
\begin{itemize}
\item[(i)]  It follows from Lemma \ref{lem:Ldecomp} and its proof 
(using the subsemilattice $L=\ctopu(G)\setminus\{\tau_G\}$), that $\fsal(G)$ admits decomposition into
translation-invariant subalgebras:
\[
\fsal^{\tau_{ap}}\oplus_{\ell^1}\left[
\ell^1\text{-}\bigoplus_{\substack{\tau_0\in\ctopu(G) \\ \tau_{ap}\subsetneq\tau_0\subsetneq\tau_G}}\left(
\fsal^{\tau_0}\cap\bigcap_{\substack{\tau\in\ctopu(G)\\ \tau_{ap}\subseteq \tau\subsetneq \tau_0}}
\ideal^\tau\right)\right]\oplus_{\ell^1}\ideal(G)
\]
where $\ideal(G)$ is defined in Remark \ref{rem:rajchmann}. 

\item[(ii)]  In the example of Theorem \ref{theo:rundes} we have that 
\[
\fsal(G_p)=\fsal^{\tau_{ap}}\oplus_{\ell^1}\ideal^{\tau_{ap}}=
\fal(\Oh_p^\times)\circ q\oplus_{\ell^1}\fal(G_p)
\]
where $q:G_p\to \Oh_p^\times$ is the quotient map, and, in fact, the map into the almost periodic compactification.
Here $G_p^\vpi=G_p\sqcup \Oh_p^\times$ is a Clifford semigroup with $2$ idempotents.  As shown in \cite{rundes1}, $\mathrm{AM}_{op}(\fsal(G_p))=5$.

\item[(ii)]  We have that $G$ is compact if and only if $\ctopu(G)=\{\tau_G\}$,
in particular exactly when $\tau_G=\tau_{ap}$.  As shown in \cite{rundes},
this is exactly the case when $\mathrm{AM}_{op}(\fsal(G))<5$.
\end{itemize}
\end{remark}

\section{Connected groups}

As we saw in the last section, we are interested in locally compact groups $G$ for which
$\ze(G^\vpi)$ is finite.  If $G$ is connected, it is shown by Ruppert~\cite[5.2]{ruppertB}
that idempotents in $G^\vpi$ are automatically central.  Hence for a connected group
we shall consider the set $\ee(G^\vpi)$ of all idempotents in $G^\vpi$.

We call $G$ {\it totally minimal} if any quotient by a closed normal subgroup $G/S$, admits
no Hausdorff group topology which is strictly coarser that the quotient topology.
In particular, for a totally minimal group locally compact group $G$ we have that
\[
\topu(G)=\{\tau_{G:S}:S\text{ is a closed normal subgroup of }G\}
\]
where each $\tau_{G:S}=\{q_S^{-1}(U):U\in\tau_{G/S}\}$.  Here, and hereafter, we shall
always let $q_S:G\to G/S$ denote the quotient map.  If we consider for closed normal subgroup
a topological {\it commensurability} relation: $S\sim S'$ if and only if $S/(S\cap S')$ and $S'/(S\cap S')$
are compact, then we may consider certain minimal representative of each commensurability class
\[
\underline{S}=\bigcap\{S':S'\text{ is a closed normal subgroup of }G\text{ with }S'\sim S\}
\]
As shown in \cite[\S 4.3]{spronk}, $S/\underline{S}$ is compact.
Letting $\fN(G)$ denote the collection of these distinct minimal representatives we see that
\[
\ctopu(G)=\{\tau_{G:S}:S\in\fN(G)\}.
\]

\begin{theorem}\label{theo:mayer} {\rm (Mayer \cite{mayer1,mayer2})}
Let $G$ be a connected locally compact group.  Consider the condition
\begin{itemize}
\item[(i)] $\ee(G^\vpi)$ is finite.
\end{itemize}
This condition implies the following two conditions, which are equivalent to one another.
\begin{itemize}
\item[(ii)] $G$ is totally minimal; and

\item[(iii)] there is a compact subgroup $M$ for which
\[
G/M=N\rtimes R
\]
where $N$ is a connected nilpotent Lie group and $R$ is a connecte linear reductive group, which acts
on $N$ without non-trivial fixed points.
\end{itemize}
\end{theorem}

\begin{proof}  The equivalence of (ii) and (iii) is \cite[Theo.\ 2.4]{mayer1}.

The proof that (i) implies (ii) is essentially \cite[Rem.\ 18]{mayer2}.  However,
that result is given in the context of the weakly almost periodic compactification,
so we shall check that it works in the context of the Eberlein compactification, $G^\vpi$, too.

Let given any closed normal subgroup $S$, the fact that
$\fsal(G:S)=\fsal(G/S)\circ q_S$ is closed subalgebra of $\fsal(G)$ (\cite{eymard, arsac}) shows that
the uniform closures $\fE(G:S)=\fE(G/S)\circ q_S$ and
$\fE(G)$ enjoys the same containment relation.  Hence restricting characters from $\fE(G)$
to $\fE(G:S)$ gives a homomorphism $\vpi_G(x)\mapsto \vpi_{G/S}(xS)$
which extends to a semigroup quotient map from $G^\vpi$ to $(G/S)^\vpi$.

If the centre $Z(G/S)$ is not compact, then for some Lie quotient $L=(G/S)/M$
by a compact subgroup $Z(L)$ is not compact.  
But then $Z(L)$ contains a discrete copy of integers, $\Zee$, hence so to must
$Z(G/S)$.  Thanks to Cowling and Rodway~\cite{cowlingr}, $\fsal(G/S)|_{\Zee}=\fsal(\Zee)$,
so $\fE(G/S)|_{\Zee}=\fE(\Zee)$.  But $\ee(\Zee)$ is infinite, thanks to Bouziad et al~\cite{bouziadlm}
(see, also, \cite{elgun}).  Hence, since any closed subsemigroup of a semitopological
semigroup must contain an idempotent, we see that $\ee((G/S)^\vpi)$, and hence $\ee(G^\vpi)$,
are each infinite.   

Hence each $Z(G/S)$ is compact, and (ii) follows form \cite[Prop.\ 2.4]{mayer1}.
\end{proof}

\begin{remark}\label{rem:FSconnmin}
It is amusing to see what the structure of $\fsal(G)$ is for a connected totally
minimal group $G$.  The decomposition below is merely a reformulation of
one of Mayer~\cite[Theo.\ 15]{mayer2}.  However, in learning about it, we shall develop
some tools for later use.

Given $G$, $M$, $N$ and $R$ as in the theorem above, let $\wbar{G}=G/M$.
We first consider the family of minimal non-commensurable normal subgroups 
of the Lie quotient, $\fN(\wbar{G})$.

We first consider subgroups of $N$.  Let
\[
\fN_R(N)=\{S\subseteq N:S\text{ is a connected normal }R\text{-invariant subgroup}\}.
\]
The see that such elements are non-commensurable, we appeal to the structure theory
of nilpotent groups (e.g.\ \cite[11.2.2]{hilgertnB}) to see that an compact subgroup $T$
of $N$ must be central.  Suppose $T$ is the maximal compact group, hence characteristic
and thus $R$-invariant.
But connected $R$ acts on $T$, hence on its discrete dual
without non-trvial fixed points, so $T$ is trivial.  Thus we conclude that
the exponential mapping from the Lie algebra, $\exp:\ffn\to N$ is a homeomorphism.
Hence if $S\supsetneq S'$ in $\fN_R(N)$, then $S/S'$ is be homeomorphic
to the quotient Lie algebra $\ffs/\ffs'$.

We may decompose $R=(H_1\times \dots\times H_n\times K)/C$ where each $H_k$ is
a non-compact simple linear Lie group, $K$ is a compact linear group, and $C$ is a finite
central subgroup of the product.  We have that $\fN(R)$ consists of groups
$H_F=\prod_{k\in F}H_k/C_F$ where $F$ is a (possibly empty) subset of $\{1,\dots,n\}$
and $C_F=C\cap\prod_{k\in F}H_k$.  We have that
\[
\fN(\wbar{G})=\bigl\{S\rtimes H_F:S\in\fN_R(N)\text{ and }
F\subseteq\{1,\dots,n\}\bigr\}.
\]

Just as in \cite[Lem.\ 4.4]{mayer1} we conclude that that $G^\vpi\cong
\bigsqcup_{S\in\fN(\wbar{G})}G/S$, and that
\begin{equation}\label{eq:FSminconn}
\fsal(\wbar{G})=\ell^1\text{-}\bigoplus_{S\in\fN(\wbar{G})}\fsal_0(\wbar{G}/S)\circ q_S.
\end{equation}
Comparing with Remark \ref{rem:FSfinidem}, and inductively using Remark \ref{rem:rajchmann}, we see that each
\[
\fsal^{\tau_{G:S_0}}\cap\bigcap_{\substack{S\in\fN(\wbar{G})\\ S_0\subsetneq S\subseteq N\rtimes H_{\{1,\dots,n\}}
}}\ideal^{\tau_{G:S}}
=\fsal_0(\wbar{G}/S_0)\circ q_{S_0}.
\]
Notice that $\tau_{G:N\rtimes H_{\{1,\dots,n\}}}=\tau_{ap}$.  In the case that $N$ is trivial, this follows form
the description of $\fsal(\wbar{G})$ for semisimple $\wbar{G}$ with finite centre, due to Cowling~\cite{cowling}.

Let us consider the decomposition of $\fsal(G)$.  For each $S$ in $\fN(\wbar{G})$ let
\begin{equation}\label{eq:tS}
\wtil{S}=\bigcap\left\{S'\subseteq G:\begin{matrix} S'\text{ is a closed normal sub-} \\
\text{group with }q_M(S')=S\end{matrix}\right\}.
\end{equation}

For $S$ in $\fN_R(N)$, let us see that $\wtil{S}$ is nilpotent.  We let
$J=r(q_M^{-1}(S)_e)$, which is the pro-solvable radical of the connected component
of $q_M^{-1}(S)$.  [In \cite[Prop.\ 2.4]{mayer1} and comments afterwards, $\wbar{G}$ is arranged by first gaining a Lie
quotient $G/L$ by a small normal subgroup, and then applying the adjoint map, which has kernel $Z(G/L)$.
We cannot ensure that this kernel (centre) is is connected, so we cannot ensure that $M$ is connected.]
By \cite[10.25]{hofmannmB} the radical is characteristic, and $q_M(J)$ is dense in $S$.  Since $q_M$ has 
compact kernel, we see that $q_M(J)=S$.  Let $S=S_1\supsetneq S_2\supsetneq\dots \supsetneq S_c\supsetneq\{e\}$
be the descending central series of $S$.  Then since $J/(J\cap M)= S$, we have that $J_{c+1}\subseteq M$,
and hence is a solvable, connected compact group, thus abelian.  Moreover, connected $J$
acts by conjugation on abelian $J_{c+1}$ trivially [just like the proof for $T\subseteq N$, above], so $J_{c+1}$ is central.
Thus $J$ is nilpotent, so $\wtil{S}\subseteq J$ is nilpotent.  

For each $F\subseteq\{1,\dots,n\}$, $q_M^{-1}(H_F)$, being a compact extension of a semi-simple group 
will contain a semi-simple group
$\wtil{H}_F=\prod_{k\in F}H_k/D_F$, where $D_F$ is a subgroup of $C_F$, above.  Moreover
$q_M(\wtil{H}_F)=H_F$.  

We obtain that
\[
\fN(G)=\{\wtil{S}\rtimes \wtil{H}_F\text{ normal in }G:S\in\fN_R(N),F\subseteq\{1,\dots,n\}\}
\]
since each $\fsal_0(G/(\wtil{S}\rtimes \wtil{H}_F))$, averages over $M$ to give
$\fsal(G/(S\rtimes H_F))$.
We may decompose $\fsal(G)$ as in (\ref{eq:FSminconn}), accordingly.
\end{remark}

We now turn to the main result of this section. 

\begin{theorem}\label{theo:BGoaGconn}
Suppose that $G$ is connected. Then $\fsal(G)$ is operator amenable if and only if $G$ is compact.
\end{theorem}

\begin{proof}
We recall Ruan's result, Theorem \ref{theo:johruan} (i) $\Leftrightarrow$ (iii): $\fal(G)$
is operator amenable exactly when $G$ is amenable.  

Hence if $G$ is compact,  $\fsal(G)=\fal(G)$ is operator amenable.  

Now, and for the rest of the proof, assume that $\fsal(G)$ is operator amenable.
We shall liberally appeal to the operator amenability properties discussed in Section \ref{ssec:bandn}.
Then we have that the completely spatially complemented ideal $\fal(G)$ is operator amenable, so $G$
is amenable.  By Theorems \ref{theo:zeEcfin} and \ref{theo:mayer} we see that
for some compact normal subgroup $M$ of $G$, $G/M=N\rtimes R$, where
$R$ is hence a quotient of $G$.  An amenable reductive group is necessarily compact
(see, for example, \cite[3.3.2]{greenleafB}).  Thus we have that
\[
\wbar{G}=G/M=N\rtimes K
\]
where $K$ is a compact connected matrix group acting on a connected nilpotent Lie group
with no non-trivial fixed points.  We shall show that $N$ must be trivial.  To the contrary, let us suppose it is not.

We let $N=N_1\supsetneq\dots\supsetneq N_d\supsetneq\{e\}$ be the descending central series.
Hence $N_d\subseteq Z(N)$.  As noted in Remark \ref{rem:FSconnmin}, $N_d$ admits no
non-trivial compact subgroups, so $N_d$ is a vector group.  
We let $V\subseteq N_d$ be of minimal dimension amongst non-trivial $K$-invariant connected subspaces
of $N_d$, so $K$ acts on $V$ irreducibly.    We let $\wtil{V}$ be defined just as in (\ref{eq:tS}).
Just as is shown in the paragraph immediately following (\ref{eq:tS}), $\wtil{V}$ is a nilpotent
group with descending central series $\wtil{V}=\wtil{V}_1\supsetneq\wtil{V}_2\supsetneq\{e\}$.  We have that
\[
V=\wtil{V}/(\wtil{V}\cap M)=(\wtil{V}/\wtil{V}_2)/[(\wtil{V}\cap M)/\wtil{V}_2]
\]
so the connected abelian group $\wtil{V}/\wtil{V}_2$ admits structure $V\times T'$, where $T'$ is compact.
But then $W=q_W^{-1}(T')$ is the maximal compact normal subgroup of $\wtil{V}$, thus characteristic in $\wtil{V}$,
and hence central in $G$ [as connected $G$ acts on it trivially, by conjugation].
Moreover, we see that $\wtil{V}$
is a central extension of $V$: $\wtil{V}/W=V$.

We let 
\[
\fsal^G(\wtil{V})=\{u\in\fsal(G)\,|\,x\mapsto u(x\cdot x^{-1}):G\to\fsal(\wtil{V})\text{ is continuous}\}
\]
By Cowling and Rodway~\cite{cowlingr} we see that $B^G(\wtil{V})=\fsal(G)|_{\wtil{V}}$.
This algebra, necessarily being a complete quotient of operator amenable
$\fsal(G)$ (the adjoint of the restriction map
is an injective $*$-homomorphism, hence a complete isometry) is also operator amenable.
We let
\[
\fsal_0^G(\wtil{V})=\fsal_0(\wtil{V})\cap \fsal^G(\wtil{V}).
\]
This, being a completely spatially complemented ideal of $\fsal^G(\wtil{V})$, 
must also be operator amenable.  In particular
\[
\fsal_0^G(\wtil{V})\text{ admits a bounded approximate identity }(u_\alp).
\]
This is the only fact about $\fsal_0^G(\wtil{V})$, which we shall use for the remainder of the proof.

Let $\wbar{K} =G/Z_{\wtil{V}}(G)$, where $Z_{\wtil{V}}(G)$ is the centralizer of $\wtil{V}$
in $G$.  Since $W$ is central in $G$ we see that $\wbar{K}$ acts on $\wtil{V}/W$
just as $K$ acts on $V$.  Hence $\wbar{K}$ is isomorphic to the image of $K$ in $\mathrm{GL}(V)$.
It then follows that
\begin{equation}\label{eq:keyidentity}
\fsal_0^G(\wtil{V}:W)=\fsal^K_0(V)
\end{equation}
where the first space is the subspace of $\fsal_0^G(\wtil{V})$ of functions constant on cosets.

The map $v\mapsto m_W\ast v:\fsal_0^G(\wtil{V})\to \fsal_0^G(\wtil{V}:W)$, 
of averaging $v$ over cosets is an expectation onto its co-domain.  Hence if
$v_\alp=m_W\ast u_\alp$, then $(v_\alp)$ is a bounded approximate identity
for $\fsal_0^G(\wtil{V}:W)$.  Indeed, for $v$ in $\fsal_0^G(\wtil{V}:W)$,
$v_\alp v=m_W\ast(u_\alp v)$ tends to $m_W\ast v=v$.
Hence it follows (\ref{eq:keyidentity}) that $\fsal_0^K(V)$ admits a bounded approxmite
identity $(v_\alp)$.  Likewise, let $M_Kv=\int_K v(k\cdot)\,dk$, and $w_\alp=M_Kv_\alp$
forms a bounded approximate identity for the fixed point space 
\[
\fsal_0(V)^K=\{w\in \fsal_0(V):w(k\cdot)=w\text{ for each }k\text{ in }K\}.
\]

The Fourier-Steiltjes transform provides an isomorphism $\fsal_0(V)^K\cong\bm_0(\what{V})^K$,
where the latter space is the space of Rajchman measures on the dual $\what{V}$ fixed by the adjoint action of $K$.
Since $K$ acts irreducibly on $V$, hence on $\what{V}$, it is a result of Ragozin~\cite{ragozin} 
that
\[
[\bm_0(\what{V})^K]^{\ast\dim V}\subseteq \bl^1(\what{V}).
\]
Hence we see that
\begin{equation}\label{eq:powerfpalg}
[\fsal_0(V)^K]^{\dim V}\subseteq \fal(V).
\end{equation}
However, the basic $K$-invariant orbital measures $\delta_{\hat{v}\cdot K}=w^*\text{-}\int_K\delta_{\hat{v}\cdot k}\,dk$,
for $\hat{v}$ in $\what{V}\setminus\{0\}$,
are in $\bm_0(\what{V})^K\setminus \bl^1(\what{V})$, which shows that $\fsal_0(V)^K\supsetneq\fal(V)$.
Hence (\ref{eq:powerfpalg}) contradicts the existence of a bounded approximate identity for $\fsal_0(V)^K$,
whence for $\fsal_0^G(\wtil{V})$.

This contradiction shows that we cannot have that $\fsal(G)$ is operator amenable with $N$
non-trivial.  
\end{proof}

\begin{remark}
For groups $G$ with Lie quotients $N\rtimes K$, as discussed in the proof of Theorem \ref{theo:BGoaGconn},
$\fsal(G)$ should admit no point derivations.  Compare \cite[\S 3]{liukkonenm1} with the fact
that each $\fal(G/S)$ admits no point derivations.   

Given (\ref{eq:FSminconn}) and it analogy for $G$, 
we highly suspect that for a connected minimal group, $G$ a compact extention of $N\rtimes R$, 
$\fsal(G)$ will not admit point derivations either.

Furthermore, many of the the constituent components $B_0(G/S)\circ q_S$ are 
not square dense in themselves, hence not operator weakly amenable.  However, as mentioned before,
we do not no of a means to show that $\fsal(G)$ itself is not operator weakly amenable.
\end{remark}

We finish with one final observation.

\begin{corollary}
For any locally compact group $G$, $\fsal(G)$ is operator amenable  only if the connected component
of the identity $G_0$ is compact.  In particular, if $G$ is almost connected, then
$\fsal(G)$ is operator amenable if and only if $G$ is compact.
\end{corollary}

\begin{proof}
It is shown in \cite[Prop. 1.1]{liukkonenm} that the restriction map $u\mapsto u|_{G_0}:\fsal(G)\to\fsal(G_0)$
is surjective.  (It is literally stated restriction holds for positive definite functions, but every element 
of $\fsal(G_0)$ is a linear combination combination of $4$ positive definite functions.)  
The adjoint of restriction is an injective $*$-homomorphism form $\wstar(G_0)$ to $\wstar(G)$,  hence
a complete isometry.  Hence  restriction is a complete quotient map.  Thus if $\fsal(G)$
is operator amenable, we too must have that
$\fsal(G_0)$ is operator amenable, and we appeal to Theorem \ref{theo:BGoaGconn}.
\end{proof}

\section*{Acknowledgement}
The author is indebted to M.~Ghandehari, with whom he extensively
discussed some early approaches to proving Theorem \ref{theo:BGoaGconn}.


\end{document}